\documentclass[reqno, 12pt]{amsart}

\usepackage{amsmath,amssymb,amsthm,amsfonts,mathrsfs}
\usepackage{fullpage}
\usepackage{relsize}
\usepackage{enumerate}
\usepackage{colonequals}
\usepackage[colorlinks=true,
linkcolor=blue,
anchorcolor=blue,
citecolor=red
]{hyperref}
\allowdisplaybreaks 
\newtheorem{theorem}{Theorem}[section]

\newtheorem{corollary}[theorem]{Corollary}

\theoremstyle{definition}

\theoremstyle{remark}

\newcommand{\N}{\mathbb{N}}
\newcommand{\Z}{\mathbb{Z}}

\newcommand{\C}{\mathbb{C}}

\newcommand{\on}{\operatorname}
\newcommand{\Li}{\on{Li}}
\renewcommand{\Re}{\on{Re}}

\renewcommand{\mod}[1]{\,(\on{mod}#1)}
\newcommand{\of}[1]{\left(#1\right)}
\newcommand{\set}[1]{\left\{#1\right\}}

\usepackage{xcolor}

\author{Biao Wang}
\address{School of Mathematics and Statistics, Yunnan University, Kunming, Yunnan 650091, China}
\email{bwang@ynu.edu.cn}

\date{\today}

\makeatletter
\@namedef{subjclassname@2020}{\textup{}2020 Mathematics Subject Classification}
\makeatother

\title[Titchmarsh divisor problem]{A generalization of the Titchmarsh divisor problem}

\begin{document}
	
\begin{abstract} Let $d^{(k)}(n)$ be the $k$-free divisor function for integer $k\ge2$. Let $a$ be a nonzero integer. In this paper, we establish an asymptotic formula
\begin{equation*}
	\sum_{p\leq x} d^{(k)}(p-a) =b_k \cdot x+O\left(\frac{x}{\log x}\right)
\end{equation*}
related to the  Titchmarsh divisor problem, where $b_k$ is a positive constant dependent on $k$ and $a$. For the proof, we apply a result of Felix and show a general asymptotic formula for a class of arithmetic functions including the unitary divisor function,  $k$-free divisor function and the proper Pillai's function.
\end{abstract}

\maketitle

\section{Introduction}

Let $d(n)$ be the divisor function. In 1930, under the Generalized Riemann Hypothesis, Titchmarsh \cite{Titchmarsh1930} proved that
\begin{equation}\label{eqn_Titchmarsh}
	\sum_{p\leq x} d(p-a)= \frac{\zeta(2)\zeta(3)}{\zeta(6)}\prod_{p\mid a}\left(1-\frac{p}{p^2-p+1}\right)\cdot x+O\of{\frac{x\log\log x}{\log x}}
\end{equation}
for any fixed non-zero integer $a$, where $p$ denotes a prime number. This leads to the study of the asymptotic behavior of the summation in \eqref{eqn_Titchmarsh}, which is known  as the \textit{Titchmarsh divisor problem}. 

In 1961, Linnik \cite{Linnik1961} proved \eqref{eqn_Titchmarsh} unconditionally with his dispersion method. Later, Halberstam \cite{Halberstam1967} gave a very short proof using the Bombieri-Vinogradov theorem. Moreover, Fouvry \cite[Corollaire 2]{Fouvry1985} and Bombieri, Friedlander and Iwaniec \cite[Corollary 1]{BFI1986} improved \eqref{eqn_Titchmarsh} to
\begin{equation}
	\sum_{p\leq x} d(p-a)= cx+c_0\Li(x)+O\of{\frac{x}{(\log x)^A}}
\end{equation}
for any $A\ge1$, where $c$ and $c_0$ are constants depending only on $a$, and $\Li(x)=\int_2^x\frac{dt}{\log t}$. After these results, a number of generalizations and analogues of the Titchmarsh divisor problem have shown in the literature, see \cite{Felix2012, Pitt2013, Pollack2016, Drappeau2017, Virdol2017, Xi2018, DrappeauTopacogullari2019, ABL2021} and so on.

In the analytic number theory, there are several arithmetic functions akin to the divisor function, say the unitary divisor function $d^\ast(n)$. Here, $d^\ast(n)$ is the number of unitary divisors of $n$, i.e., 
$d^\ast(n)=\sum_{ab=n, (a,b)=1}1$. A fact is that $d^\ast(n)$ is equal to the number of squarefree divisors of $n$ and $d^\ast(n)=2^{\omega(n)}$, where $\omega(n)$ is the number of distinct prime factors of $n$. We know that the Dirichlet series of the divisor function $d(n)$ is $\zeta^2(s)$, where $\zeta(s)=\sum_{n=1}^\infty \frac{1}{n^s}$ is the Riemann zeta functionn for $\Re s>1$.
  The Dirichlet series of $d^\ast(n)$ is closely related to $\zeta^2(s)$ and equals
\begin{equation}\label{eqn_unitary}
	\sum_{n=1}^\infty \frac{d^\ast(n)}{n^s}=\frac{\zeta^2(s)}{\zeta(2s)}.
\end{equation}

In general, let $d^{(k)}(n)$ be the number of $k$-free divisors of $n$ for integer $k\ge 2$. Then the  Dirichlet series of  $d^{(k)}(n)$ is equal to
\begin{equation}\label{eqn_k-free}
	\sum_{n=1}^\infty \frac{d^{(k)}(n)}{n^s}=\frac{\zeta^2(s)}{\zeta(ks)}.
\end{equation}

The asymptotic formulas for $\sum_{n\leq x} d^\ast(n)$ and $\sum_{n\leq x} d^{(k)}(n)$ have been widely studied in the literature, see \cite{Mertens1874, Cohen1960, GioiaVaidya1966,SubbaraoSuryanarayana1978,SuryanarayanaPrasad1971,Baker1996,Kumchev2000,FuruyaZhai2008} and so on. They are closely related to the Dirichlet divisor problem. However, to the best of our knowledge, the analogue of the Tichmarsh divisor problem for  the $k$-free divisor function has not been studied yet. In this paper, we will give a new generalization of Tichmarsh's result \eqref{eqn_Titchmarsh} and then provide an asymptotic formula for 
$$\sum_{p\leq x} d^{(k)}(p-a).$$

To unify \eqref{eqn_unitary} and \eqref{eqn_k-free}, we consider general arithmetic functions $f:\N\to \C$ with the Dirichlet series $F(s)=\sum_{n=1}^\infty\frac{f(n)}{n^s}$ absolutely convergent on $\Re s>1$, then 
$$\sum_{n=1}^\infty \frac{(f\ast d)(n)}{n^s}=F(s)\zeta^2(s),$$
where $f\ast d$ is the Dirichlet convolution between $f$ and $d$.
The following is our main result.

\begin{theorem}
\label{mainthm}
Let $f:\N\to\C$ be an arithmetic function. Let $a$ be an integer.   If the Dirichlet series of $f$ is absolutely convergent on $\Re s\geq 1-\alpha$ for some positive $\alpha >0$, then there exists some constant $c_f>0$ such that
\begin{equation}
\label{eqn_mainthm}
\sum_{p\leq x} (f\ast d)(p-a) =c_f \cdot x+O\left(\frac{x}{\log x}\right), 
\end{equation}
where
$c_f=\sum_{m=1}^\infty \frac{f(m)c_m}{m}$, and $c_m$ is defined in Theorem~\ref{thm_Felix2012} below.
The implied constant in \eqref{eqn_mainthm} depends only on $f$ and $a$.
\end{theorem}

Taking $f=1_{n=1}$,  we get \eqref{eqn_Titchmarsh} from \eqref{eqn_mainthm}. Hence \eqref{eqn_mainthm} is a generalization of \eqref{eqn_Titchmarsh}. For $k\geq2$, if we take
$$f(n)=\begin{cases}
	\mu(m) & \text{if } n=m^k;\\
	0&\text{otherwise,}
\end{cases}$$
then we get that $\sum_{n=1}^\infty \frac{f(n)}{n^s}=\frac1{\zeta(ks)}$ and $f\ast d=d^{(k)}$. In particular, for $k=2$, we get $(f\ast d)(n)=2^{\omega(n)}$. Clearly, the Dirichlet series of $f(n)$ is absolutely convergent on $\Re s>1/k$. Thus, by Theorem~\ref{mainthm}, we obtain an analogue of \eqref{eqn_Titchmarsh} for the $k$-free divisor function and the unitary divisor function as follows.

\begin{corollary}
\label{maincor}
Let $k\ge2$ be an integer. Let $a$ be an integer. Then there exists some constant $b_k>0$ such that
\begin{equation}
\sum_{p\leq x} d^{(k)}(p-a) =b_k \cdot x+O\left(\frac{x}{\log x}\right),
\end{equation}
where
$$
b_k=\frac{\zeta(2)\zeta(3)}{\zeta(6)}\prod_{p\mid a}\left(1-\frac{p}{p^2-p+1}\right) \prod_p \left(1-\frac{1}{p^{k-2}(p^2-p+1)}\right).
$$

In particular, for $\omega(n)$, the number of distinct prime factors of $n$,  we have
\begin{equation}
\sum_{p\leq x} 2^{\omega(p-a)} =b_2 \cdot x+O\left(\frac{x}{\log x}\right),
\end{equation}
where
$$
b_2=\frac{\zeta(2)\zeta(3)}{\zeta(6)}\prod_{p\mid a}\left(1-\frac{p}{p^2-p+1}\right) \prod_p \left(1-\frac{1}{ p^2-p+1}\right).
$$
\end{corollary}

In Section~\ref{sec_proof}, we will apply a result of Felix   \cite{Felix2012} in 2012 on generalizing the Titchmarsh divisor problem to give a proof of Theorem~\ref{mainthm}. Then in Section~\ref{sec_problems}, we will propose some problems related to our generalization.

\section{Proof of Theorem~\ref{mainthm}}
\label{sec_proof}

 To prove Theorem~\ref{mainthm}, we cite a generalization of \eqref{eqn_Titchmarsh} obtained by Felix \cite{Felix2012} for primes in arithmetic progressions. This result helps us  divide the summation in \eqref{eqn_mainthm} into two parts. The main part consists of summations over small  moduli, which involves the main term in \eqref{eqn_mainthm}.  The  other part consisting of summations over large moduli contributes only an error term in \eqref{eqn_mainthm}.
 
 \begin{theorem}[{\cite[Theorem 1.2]{Felix2012}}]
\label{thm_Felix2012}
Let $m>1$ be an integer, let $a\in\Z$ and $(a,m)=1$, and let $A>0$. Then we have 
\begin{equation}
	\sum_{\substack{p\leq x\\ p\equiv a\mod m}} d\left(\frac{p-a}{m}\right) =\frac{c_m}{m} x +O\of{\frac{x(\log m)(1+c_m)}{m\log x}} + O\of{\frac{x}{(\log x)^A}}
\end{equation}
uniformly for $m\leq (\log x)^{A+1}$, where
$$c_m=\frac{\zeta(2)\zeta(3)}{\zeta(6)}\prod_{p\mid a}\left(1-\frac{p}{p^2-p+1}\right) \prod_{p\mid m}\left(1 + \frac{p-1}{p^2-p+1}\right)$$
and the first $O$-constant is absolute and the second  $O$-constant depends only on $a$ and $A$.
\end{theorem}

\begin{proof}[Proof of Theorem~\ref{mainthm}]
Let $B>0$ be a parameter to be chosen later. By $(f\ast d)(n)=\sum_{mq=n}f(m)d(q)$, we have
\begin{align}
	&\quad\sum_{p\leq x} (f\ast d)(p-a) = \sum_{p\leq x} \sum_{mq=p-a}f(m)d(q) \nonumber\\
	&= \sum_{m\leq x-a}f(m) \sum_{\substack{p\leq x\\ mq=p-a}}  d(q) = \sum_{m\leq x-a}f(m) \sum_{\substack{p\leq x\\ p\equiv a\mod{m}}}  d\of{\frac{p-a}{m}} \nonumber\\
	&= \sum_{m\leq (\log x)^B} f(m) \sum_{\substack{p\leq x\\ p\equiv a\mod{m}}}  d\of{\frac{p-a}{m}}  + \sum_{(\log x)^B <m\leq x-a}f(m) \sum_{\substack{p\leq x\\ p\equiv a\mod{m}}}  d\of{\frac{p-a}{m}} \nonumber\\
	&:=S_1+S_2. \label{eqn_mainthm_pf_S1S2}
\end{align}

For $S_1$, by Theorem~\ref{thm_Felix2012}, we have
\begin{align}
	S_1&=\sum_{m\leq (\log x)^B} f(m) \sum_{\substack{p\leq x\\ p\equiv a\mod{m}}}  d\of{\frac{p-a}{m}}\nonumber\\
	&=\sum_{m\leq (\log x)^{B}}f(m)\left\{ \frac{c_m}{m} x +O\of{\frac{x(\log m)(1+c_m)}{m\log x}} + O\of{\frac{x}{(\log x)^{2B}}} \right\} \nonumber\\
	&=\left(\sum_{m\leq (\log x)^{B}} \frac{f(m)c_m}{m} \right) x +O\of{\sum_{m\leq (\log x)^{B}}\frac{x |f(m)|(\log m)(1+c_m)}{m\log x}} \nonumber\\ 
	&\qquad + O\of{\frac{x}{(\log x)^{2B}} \sum_{m\leq (\log x)^{B}}|f(m)|}.  \label{eqn_mainthm_pf_S1-1}
\end{align}

Notice that $$c_m\ll \prod_{p|m}(1+\frac1p)\ll \log m, $$
where the implied constant depends only on $a$. Then for the first term in \eqref{eqn_mainthm_pf_S1-1}, we have
\begin{equation*}
	\sum_{m\leq (\log x)^{B}} \frac{f(m)c_m}{m}= \sum_{m=1}^\infty \frac{f(m)c_m}{m}-\sum_{m > (\log x)^{B}} \frac{f(m)c_m}{m}
\end{equation*}
and
\begin{align*}
	&\quad\sum_{m > (\log x)^{B}} \frac{f(m)c_m}{m}\ll  \sum_{m> (\log x)^{B}} \frac{|f(m)| \log m}{m}\\
	&\ll  \sum_{m> (\log x)^{B}} \frac{|f(m)| \log m}{m} \cdot \frac{m^{\alpha/2}}{(\log x)^{\alpha B/2}} \ll \frac{1}{(\log x)^{\alpha B/2}}.
\end{align*}
Here $\sum_{m=1}^\infty \frac{|f(m)| \log m}{m^{1-\alpha/2}}$ is convergent due to the assumption that the Dirichlet series of $f$ is absolutely convergent at $1-\alpha$. So
\begin{equation} \label{eqn_mainthm_pf_S11}
	\sum_{m\leq (\log x)^{B}} \frac{f(m)c_m}{m}= \sum_{m=1}^\infty \frac{f(m)c_m}{m}+O\of{\frac{1}{(\log x)^{\alpha B/2}}}.
\end{equation}

For the second term in \eqref{eqn_mainthm_pf_S1-1},  we have
\begin{equation} \label{eqn_mainthm_pf_S12}
	\sum_{m\leq (\log x)^{B}}\frac{x |f(m)|(\log m)(1+c_m)}{m\log x}\ll \frac{x}{\log x}  \sum_{m\leq (\log x)^{B}}\frac{ |f(m)|(\log m)^2}{m} \ll \frac{x}{\log x}.
\end{equation}

For the third term in \eqref{eqn_mainthm_pf_S1-1}, we have
\begin{equation} \label{eqn_mainthm_pf_S13}
	\frac{x}{(\log x)^{2B}} \sum_{m\leq (\log x)^{B}}|f(m)| \leq \frac{x}{(\log x)^{2B}} \left(\sum_{m\leq (\log x)^{B}}\frac{|f(m)|}{m} \cdot (\log x)^{B}\right)\ll \frac{x}{(\log x)^{B}}.
\end{equation}

Thus, by \eqref{eqn_mainthm_pf_S1-1}-\eqref{eqn_mainthm_pf_S13}, we obtain that
\begin{equation}\label{eqn_S1}
	S_1=\left(\sum_{m=1}^\infty \frac{f(m)c_m}{m}\right)\cdot x + O\of{\frac{x}{\log x}},
\end{equation}
provided $B\geq\max\set{1,2/\alpha}$.

For $S_2$, we have
\begin{align}
	|S_2|&\leq \sum_{(\log x)^{B}< m\leq x-a}  |f(m)|\sum_{\substack{p\leq x\\ mq=p-a}}  d(q) \nonumber\\
	&\leq  \sum_{m> (\log x)^{B}} |f(m)| \sum_{q\leq \frac{x-a}{m}} d(q) \nonumber\\
	&\ll \sum_{m> (\log x)^{B}} |f(m)| \cdot\frac{x}{m}\log\frac{x}{m} \nonumber\\
	&\ll x\log x \sum_{m> (\log x)^{B}} \frac{|f(m)|}{m}. \label{eqn_mainthm_pf_S2-1}
\end{align}
We estimate the last summation in \eqref{eqn_mainthm_pf_S2-1} as follows:
\begin{equation*}
	\sum_{m> (\log x)^{B}} \frac{|f(m)|}{m}\leq \sum_{m> (\log x)^{B}} \frac{|f(m)|}{m} \cdot \frac{m^\alpha}{(\log x)^{\alpha B}}\ll \frac{1}{(\log x)^{\alpha B}}.
\end{equation*}
It follows that
\begin{equation}\label{eqn_S2}
	S_2=O\of{\frac{x}{\log x}},
\end{equation}
provided $B\ge2/\alpha$. 

Therefore, taking $B\geq\max\set{1,2/\alpha}$, we obtain \eqref{eqn_mainthm} by combining \eqref{eqn_mainthm_pf_S1S2}, \eqref{eqn_S1}, and \eqref{eqn_S2} together. 	
\end{proof}

\section{Related problems}
\label{sec_problems}

Let $P(n)=\frac1n\sum_{k=1}^n\gcd(k,n)$ be the proper Pillai's function, then $P(n)=\sum_{m \mid n} \varphi(m)/m$, where $\varphi(m)$ is the Euler totient function, and
\begin{equation}
	\sum_{n=1}^\infty \frac{P(n)}{n^s}=\frac{\zeta^2(s)}{\zeta(s+1)}
\end{equation}
for $\Re >1$, see \cite{KNT2012}. The analysis in Theorem~\ref{mainthm} shows that there exists some constant $c_P>0$ such that
\begin{equation}
	\sum_{p\leq x}P(p-a)=c_P\cdot x+O\of{\frac{x}{\log x}}.
\end{equation}
This is another application of  Theorem~\ref{mainthm}. 

In the literature, there are a number of generalizations and analogues of the Titchmarsh divisor problem.  In the end, motivated by the asymptotic formula in Theorem~\ref{mainthm}, we propose a list of problems related to these generalizations and analogues for the interested readers. Let $f:\N\to \C$ be an arithmetic function satisfying that its Dirichlet series is absolutely convergent on $\Re s\ge 1-\alpha$ for some $\alpha>0$. Let $g=f\ast d$. For example, one may take $g=2^{\omega(n)}, d^{(k)}(n)$ or $P(n)$ as in this paper. The following is our problem list.
\begin{enumerate} 
	\item Estimate the summation
	$$\sum_{p_1p_2\cdots p_k \leq x}g(p_1p_2\cdots p_k-1),$$
	c.f. \cite{KS1968, Fujii1976}.
	
	\item Find an analogue of Corollary~\ref{maincor} over finite fields, c.f. \cite{ABR2015, DeySavalia2023}. 
	\item Under the Elliott-Halberstam conjecture, estimate
	$$\sum_{\substack{p\leq x\\ P^+(p-a)\ge y}}g(p-a),$$
	where $P^+(n)$ denotes the largest prime factor of $n$,	c.f. \cite{Wu2018}.
	\item Estimate  the summation
	$$\sum_{p\leq x}g(p^2+1),$$
		c.f. \cite{Xi2018}.
	\item Estimate  the summation
	$$\sum_{p^2+q^2\leq x}g(p^2+q^2+1),$$
	where $p,q$ belong to the set of primes.	 c.f. \cite{Li2020}.

\end{enumerate}


\end{document}